\documentclass[a4paper,oneside,12pt,reqno]{amsart}
\usepackage[utf8]{inputenc}
\usepackage{amsaddr}

\usepackage{xcolor, graphicx}
\usepackage{mathrsfs, amsthm, amsmath, amssymb, mathtools, bbm}
\usepackage{hyperref}
\hypersetup{colorlinks=true, linkcolor=purple, citecolor=cyan}
\usepackage{geometry}
\usepackage{enumitem}
\usepackage{comment}

\DeclareMathAlphabet{\mathmybb}{U}{bbold}{m}{n}

\newtheorem{thm}{Theorem}
\newtheorem{lemma}[thm]{Lemma}

\newtheorem{prop}[thm]{Proposition}

\newcommand{\tuple}[1]{{\bf #1}}

\newcommand{\close}{\sigma}
\newcommand{\open}{\sigma^{-1}}
\newcommand{\merge}{\varphi}
\newcommand{\sepa}{\psi}
\newcommand{\sew}{\widetilde{\varphi}}
\newcommand{\cut}{\widetilde{\psi}}
\newcommand{\tree}{\tau}
\newcommand{\optree}{\theta}

\newcommand{\gm}{\mathfrak{m}}
\newcommand{\gt}{\mathfrak{t}}
\newcommand{\pp}{\mathbbm{p}}
\newcommand{\cc}{\mathbbm{c}}
\newcommand{\cB}{\mathcal{B}}
\newcommand{\cE}{\mathcal{E}}
\newcommand{\cT}{\mathcal{T}}

\title[Slit-Slide-Sew bijections for planar bipartite maps]{Slit-Slide-Sew bijections for planar bipartite maps with prescribed degrees}
\author{Juliette Schabanel}
\address{LaBRI, Université de Bordeaux, France}
\email{jschaban@phare.normalesup.org}

\begin{document}

\begin{abstract}
    We present a bijective proof for the planar case of Louf's counting formula on bipartite planar maps with prescribed face degree, that arises from the Toda hierarchy. We actually show that Louf's formula hides two simpler formulas, both of which can be rewritten as equations on trees using duality and Schaeffer’s bijection for Eulerian maps. We prove them bijectively and show that the constructions we provide for trees can also be interpreted as ``slit-slide-sew” operations on maps. As far as we know, this is the first bijection for a formula with infinitely many parameters arising from an integrable hierarchy.
\end{abstract}

\maketitle

\section*{Introduction}

\textbf{Context.} A planar map is a combinatorial object describing the embedding of a planar multigraph on the sphere up to homeomorphism (see Section~\ref{sec:setting} for precise definitions). Maps are known to have a lot of applications in different fields of mathematics, physics, and computer science. Hence, their structures have been widely studied during the last decades.

The enumeration of planar maps started with the work of Tutte in the 60's \cite{Tut63}. Tutte used decomposition properties of maps to write equations satisfied by the generating functions of planar maps and obtain enumeration formulas for them. The formulas he derived that way are remarkably simple, and their structure, with the presence of Catalan numbers, suggested that maps could be interpreted as some decorated trees. After an initial work of bijective explanation by Cori and Vauquelin \cite{CoVa81}, Schaeffer developed the field by obtaining various bijective constructions between maps and trees \cite{Sch97}. His work was followed by numerous papers dealing with various families of maps \cite{BDG04,BeFu12,AlPo13,Bet18}. 

Another powerful tool for map enumeration is integrable hierarchies such as the KP and the Toda hierarchies. Integrable hierarchies are infinite sets of partial differential equations on functions with an infinite number of variables $(p_i)_{i\geqslant1}$ which arose from mathematical physics. Goulden and Jackson showed in 2008 \cite{GoJa08} that certain generating functions for maps (with $p_i$ counting the vertices of degree $i$) are solution to the KP hierarchy, which allowed them to derive a very simple recurrence formula for cubic maps, i.e. maps in which all vertices have degree $3$. They were followed by Carrel and Chapuy  with a recurrence formula on general maps \cite{CaCh15} and Kazarian and Zograf for bipartite maps \cite{KaZo15}. Using similar methods with the Toda hierarchy, Louf \cite{Louf21} recently derived recurrence formulas for bipartite maps with prescribed face degrees and for constellations. 

Finding bijective explanations of these formulas would give us a deeper understanding of maps. Bijective proofs for the Goulden--Jackson and Carrel--Chapuy formulas were found by Chapuy, Feray and Fusy in the one-faced case \cite{CFF13} and by Louf in the planar case \cite{Louf19}, but the question still remains mainly an open problem.

\textbf{Contribution.} In this paper, we present a bijective proof for the planar case of Louf's formula on bipartite maps with prescribed degrees \cite{Louf21}. We actually show that Louf's formula hides two simpler formulas, both of which can be rewritten as equations on trees using duality and Schaeffer's bijection for Eulerian maps \cite{Sch97}. We prove them bijectively and show that the constructions we provide for trees can also be interpreted as ``slit-slide-sew" operations on maps, similar to those of \cite{Bet18,Louf19, BGM22, BeKo26, BeFuLo24}.
As far as we know, this is the first bijection for a formula with infinitely many parameters arising from an integrable hierarchy.

\textbf{Structure of the paper.} In Section~\ref{sec:setting}, we provide some definitions on maps and explain the formulas our bijections prove. In Section~\ref{sec:trees} recall a bijection due to Schaeffer \cite{Sch97} which allows us to express the equations in a simpler context and make the bijections easier to understand. The two bijections on bipartite maps are described respectively in Section~\ref{sec:bijection vertex} and Section~\ref{sec:bijection face}.   

\section{Setting and presentation of the results}
\label{sec:setting}

\subsection{Definitions and notations}

A \emph{planar map} is a proper embedding of a finite connected graph (possibly with loops and multiple edges) into the sphere, where \emph{proper} means that edges are smooth simple arcs which meet only at their endpoints. Two maps are identified if they can be mapped one onto the other by an orientation preserving homeomorphism. The \emph{vertices} and \emph{edges} of the maps are inherited from the graph and the \emph{faces} of the map are the connected components of the complementary of the embedded graph. We call \emph{half-edge} an edge carrying one of its two possible orientations. A half-edge $h$ is \emph{incident} to its origin vertex and to the face $f$ that lies on its left. The \emph{degree} of a vertex or a face is the number of half-edges incident to it. 


We implicitly consider our maps to be rooted on a half-edge, which means that one half-edge is distinguished. This half-edge is called the \emph{root edge} of the map and the face it is incident to is called the \emph{outer face} of the map.


If $\gm$ is a map, we denote by $V(\gm)$ the set of its vertices, by $E(\gm)$ the set of its edges and by $F(\gm)$ the set of its faces. We also define $v(\gm)=|V(\gm)|$ its number of vertices, $n(\gm)=|E(\gm)|$ its number of edges and $f(\gm)=|F(\gm)|$ its number of faces. 
Those quantities satisfy the following relation. 

\begin{thm}[Euler's formula]
    For every planar map $\gm$, $v(\gm)+f(\gm) = n(\gm)+2$.
\end{thm}

Let $\gm$ be a map. A \emph{path} in $\gm$ is a sequence $h_1 \ldots h_k$ of half-edges such that for each $i \in \{1, \ldots, k-1\}$, the end of $h_i$ is the origin of $h_{i+1}$. A \emph{cycle} is a path such that the end of $h_k$ is the origin of $h_1$. The \emph{length} of a path is the number of half-edges it contains. 

A map is \emph{bipartite} if all its cycles have even length. Bipartiteness is equivalent to the existence of a coloring of the vertices of the map with two colors, black and white, such that the ends of each edge receive different colors. A bipartite map can be endowed with a canonical coloring, the one such that the root edge is oriented from white to black. In what follows, we will assume that all our bipartite maps are equipped with this coloring. Note that the orientation of the root edge can be recovered from the coloring.

Let $\tuple{d} = (d_1, d_2, \ldots)$ be an infinite sequence of non-negative integers with finitely many nonzero terms. A bipartite map $\gm$ has \emph{degree distribution} $\tuple{d}$ if it has exactly $d_i$ faces of degree $2i$ for each $i>0$. We denote by $\mathcal{B}(\tuple{d})$ the set of planar bipartite maps of degree distribution $\tuple{d}$ and by $B(\tuple{d})$ their number with the convention $B(\tuple{0}) = 0$. Observe that the degree distribution determines the numbers of faces, edges and vertices of the map: \[f(\tuple{d}) = \sum_{i\geqslant1} d_i,\]
\[n(\tuple{d}) = \sum_{i\geqslant1} id_i\]
and Euler's formula implies \[v(\tuple{d}) = n(\tuple{d})+2-f(\tuple{d}) = 2+ \sum_{i\geqslant1}(i-1)d_i.\] 


\subsection{Recursive formulas}

Our goal is to provide a bijective proof of the following formula obtained by Louf in \cite{Louf21} for bipartite maps, in the planar case.

\begin{thm}[\cite{Louf21}]
    The numbers $B(\tuple{d})$ of planar bipartite maps 
    with degree distribution $\tuple{d}$ satisfy the following recurrence relation:
    \begin{equation}
\label{eq:origin}
    \left(\!\binom{n(\tuple{d})+1}{2}-\binom{v(\tuple{d})}{2}\!\right)B(\tuple{d}) = \sum_{\tuple{s} + \tuple{t} = \tuple{d}} (1+n(\tuple{s}))\binom{v(\tuple{t})}{2} B(\tuple{s})B(\tuple{t}).
\end{equation}
    where $\displaystyle f(\tuple{d}) = \sum_{i\geqslant1} d_i$, $\displaystyle n(\tuple{d}) = \sum_{i\geqslant1} id_i$, $\displaystyle v(\tuple{d}) = n(\tuple{d})+2-f(\tuple{d}) = 2+ \sum_{i\geqslant1}(i-1)d_i$ and $\tuple{s} + \tuple{t} = \tuple{d}$ means $s_i + t_i = d_i$ for all $i \geqslant 1$.
\end{thm}


We would like to interpret each side of \eqref{eq:origin} as the numbers of respectively one and a pair of bipartite maps with some marked elements. To do so, we first get rid of the difference in the left-hand side. Using Euler's formula, we can rewrite \eqref{eq:origin} as follows:

\begin{equation}
\label{eq:rewrite}
    \left(\!(f(\tuple{d})-1)v(\tuple{d}) + \binom{f(\tuple{d})-1}{2}\!\right)B(\tuple{d}) = \sum_{\tuple{s} + \tuple{t} = \tuple{d}} (v(\tuple{s}) + f(\tuple{s})-1)\binom{v(\tuple{t})}{2} B(\tuple{s})B(\tuple{t}).
\end{equation}

This leads to a tempting separation that turns out to be true. In the following sections, we will give a bijective proof of the following formulas for bipartite maps: 

\begin{equation}
  \tag{marked vertex identity}
  \label{eq:vertex}
  4(f(\tuple{d})-1)B(\tuple{d}) = \sum_{\tuple{s}+\tuple{t}=\tuple{d}} v(\tuple{s})v(\tuple{t}) B(\tuple{s})B(\tuple{t})
\end{equation}

\begin{equation}
  \tag{marked face identity}
  \label{eq:face}
  \binom{f(\tuple{d})-1}{2}B(\tuple{d}) = \sum_{\tuple{s}+\tuple{t}=\tuple{d}} (f(\tuple{s})-1)\binom{v(\tuple{t})}{2} B(\tuple{s})B(\tuple{t})
\end{equation}

\begin{prop}
    Equation~\eqref{eq:rewrite} is a corollary of the \ref{eq:vertex} and the \ref{eq:face}.
\end{prop}
\begin{proof}
    First multiply the \ref{eq:vertex} by $v(\tuple{d})$. Then, Euler's formula gives 
    \vspace{-0.2cm}
    \begin{align*}
        v(\tuple{d}) & = n(\tuple{d}) +2 - f(\tuple{d}) = 2 + \sum_{i\geqslant1} (i-1)d_i = 2 + \sum_{i\geqslant1} (i-1)(s_i+t_i) \\& = 2 + n(\tuple{s}) - f(\tuple{s}) + n(\tuple{t}) - f(\tuple{t}) = v(\tuple{s}) + v(\tuple{t}) -2.
    \end{align*} Thus we have 
    \begin{multline*}
    4v(\tuple{d})(f(\tuple{d})-1)B(\tuple{d}) = \sum_{\tuple{s}+\tuple{t}=\tuple{d}}v(\tuple{s})v(\tuple{t})(v(\tuple{s})-1+v(\tuple{t})-1)B(\tuple{s})B(\tuple{t})\\
    =\sum_{\tuple{s}+\tuple{t}=\tuple{d}}v(\tuple{s})v(\tuple{t})(v(\tuple{s})-1)B(\tuple{s})B(\tuple{t}) + \sum_{\tuple{s}+\tuple{t}=\tuple{d}}v(\tuple{s})v(\tuple{t})(v(\tuple{t})-1)B(\tuple{s})B(\tuple{t}).
    \end{multline*}
    By symmetry, the two sums in the right-hand side are equal. So by dividing by $4$ we obtain 
    \[v(\tuple{d})(f(\tuple{d})-1)B(\tuple{d}) = \sum_{\tuple{s}+\tuple{t}=\tuple{d}}v(\tuple{s})\binom{v(\tuple{t})}{2}B(\tuple{s})B(\tuple{t}).\]
    Summing this with the \ref{eq:face} gives \eqref{eq:origin}.
\end{proof}

\subsection{Bijective strategy}

In a different context, Bouttier, Guitter and Miermont constructed a bijection between pairs of trees with respectively $m_1$ edges and one marked vertex and $m_2$ edges and two marked vertices, and bipartites maps with exactly two faces of respective degrees $2m_1$ and $2m_2$ and one marked vertex \cite[Proposition 7]{BGM22}. Observe that this corresponds to the special case where $f(\tuple{d}) = 2$ of the equation 
\[v(\tuple{d})(f(\tuple{d})-1)B(\tuple{d}) = \sum_{\tuple{s}+\tuple{t}=\tuple{d}}v(\tuple{s})\binom{v(\tuple{t})}{2}B(\tuple{s})B(\tuple{t}),\]
which is equivalent to the \ref{eq:vertex}. 

Their method consists in slitting in each tree the same number of edges along the unique path between two marked elements, and then sew the two created cycles together to obtain a bipartite map with exactly one cycle. The reverse construction consists in cutting the two-faced map along its unique cycle and then closing the ``hole'' in each of the two resulting maps to obtain trees. 

Our bijective proof of the marked vertex identity will be a generalization of their method to bipartite maps with more than two faces, up to a few adaptations which allow us to forget the additional marked vertex. The \ref{eq:face} will then be obtained by applying three times the other bijection. 

For their method to generalize to maps with more than two faces, we need some way to make the choice of a path to slit and of a cycle to cut canonical. The simplest solution for that is to take a spanning tree of the map, and then consider the paths in that spanning tree. Of course, one cannot choose any spanning tree: it needs to have some properties that ensure that after applying a construction, the special path in the resulting map is indeed the path in the spanning tree of the map. A good choice of a spanning tree exists, and we could describe it as well as the bijection directly but many choices in the construction would appear arbitrary. Instead, we choose to do a detour on the dual of bipartite maps, Eulerian maps, on which the construction appears more naturally.

\section{Bipartite maps and Eulerian trees}
\label{sec:trees}

The goal of this section is to introduce a class of trees, \emph{balanced Eulerian trees}, which are in bijection with bipartite maps by a result of Schaeffer \cite{Sch97}, and for which our bijection will appear naturally. These trees will be trees on the faces of our maps, so before getting to them let us introduce duality on maps.

\subsection{Duality}

The \emph{dual} $\gm^*$ of a given map $\gm$ is the map obtained by putting a vertex in each face of $\gm$ and for each edge $e$ of $\gm$, an edge $e^*$ between the two faces it separates. The root of $\gm^*$ is the dual of the root of $\gm$, oriented from the left of the root of $\gm$ to its right. 
An example of a map and its dual is given in Figure~\ref{fig:dual}. 


\begin{figure}[ht!]
    \centering
    \includegraphics[height=5cm]{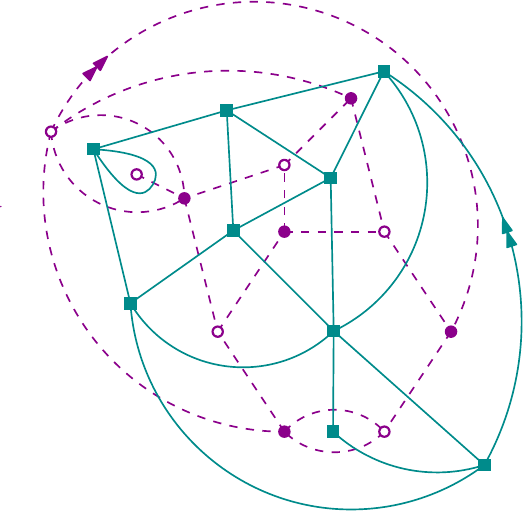}
    \caption{The blue map is the dual of the purple dashed map.}
    \label{fig:dual}
\end{figure}

\begin{prop}
    Duality is a bijection that exchanges vertices and faces, preserving their degree. 
    
    More precisely, for any planar map $\gm$, $(\gm^*)^*$ is equal to $\gm$ but with the root edge orientation reversed.
\end{prop}

Therefore, the dual of a bipartite map is a map whose vertices all have even degree. Such a map is called \emph{Eulerian}. A Eulerian map has degree distribution $\tuple{d}$ if it has exactly $d_i$ vertices of degree $2i$ for all $i>0$. We denote by $\mathcal{E}(\tuple{d})$ the set of planar Eulerian maps of degree distribution $\tuple{d}$ and by $E(\tuple{d})$ their number. 

\begin{prop}
    For all $\tuple{d}$, duality induces a bijection between $\mathcal{B}(\tuple{d})$ and $\mathcal{E}(\tuple{d})$. 
\end{prop}

\subsection{Eulerian trees}

In this subsection, we describe a bijection due to Schaeffer in \cite{Sch97} between Eulerian maps and a particular kind of trees, \emph{balanced Eulerian trees}.

A \emph{plane tree} is a planar map with only one face. Vertices of degree $1$ are called \emph{leaves} and the others \emph{inner vertices}. An edge is said to be \emph{inner} if it is between two inner vertices. We root our trees by distinguishing a leaf instead of a half-edge. 

An \emph{Eulerian tree} is a rooted plane tree with leaves colored black or white such that:
\vspace{-0.3cm}
\begin{itemize}
    \item all inner vertices have even degree,
    \item each vertex of degree $2i$ is adjacent to exactly $i-1$ white leaves,
    \item the root leaf is black.
\end{itemize}
See Figure~\ref{fig:tree-word} for an example.

Two Eulerian trees are \emph{conjugated} if one can be obtained from the other by changing which black leaf is the root.

A Eulerian tree has degree distribution $\tuple{d}$ if it has exactly $d_i$ inner vertices of degree $2i$ for all $i$. 
Notice that such a tree has $\sum_{i>0} d_i$ inner vertices, one less inner edges and $\sum_{i>0}(i-1)d_i$ white leaves. Therefore, by double counting the edges, its number of black leaves is $\sum_{i>0}(2i-i+1)d_i-2(\sum_{i>0}d_i-1) = \sum_{i>0}(i-1)d_i +2$, i.e. a Eulerian tree has exactly $2$ more black leaves than white leaves.

To define balanced trees and describe Schaeffer's bijection we introduce the word associated to a Eulerian tree. The \emph{word} associated to a Eulerian tree $\gt$ is the word $\omega(\gt)$ on the alphabet $\{b, w\}$ obtained by following the border of the tree in counterclockwise direction, ending at the root leaf, and writing $b$ when a black leaf is seen and $w$ for a white one (see Figure~\ref{fig:tree-word}). Conjugated trees lead to conjugated words, which means that if $\gt$ and $\gt'$ are conjugated, there are two words $\omega_1$ and $\omega_2$ such that $\omega(\gt)=\omega_1\omega_2$ and $\omega(\gt')=\omega_2\omega_1$.

These words can then be read as bracketing words, $b$ being the closing bracket and $w$ the opening bracket. A bracketing word is \emph{correct} if for every prefix $u$ of it, the number of letters $w$ in $u$ is greater or equal to the number of letters $b$ in $u$. In such a word, each opening bracket is closed by a closing one, we say that those brackets are \emph{matched}. This induces a (partial) matching on the leaves of the tree.  

Recall that in a Eulerian tree there are exactly two more black leaves than white leaves. It follows from the Cycle Lemma that in each conjugation class of Eulerian trees, there are exactly two rooted trees such that $\omega(\gt) = p_1bp_2b$ with $p_1$ and $p_2$ correct bracketing words. The two unmatched black leaves in a Eulerian tree are called the \emph{free leaves}.

A Eulerian tree is \emph{balanced} if it is rooted on a free leaf. We denote by $\mathcal{T}(\tuple{d})$ the set of balanced Eulerian trees of degree distribution $\tuple{d}$.

\begin{figure}[ht]
    \centering
    \includegraphics[width = 0.95\textwidth]{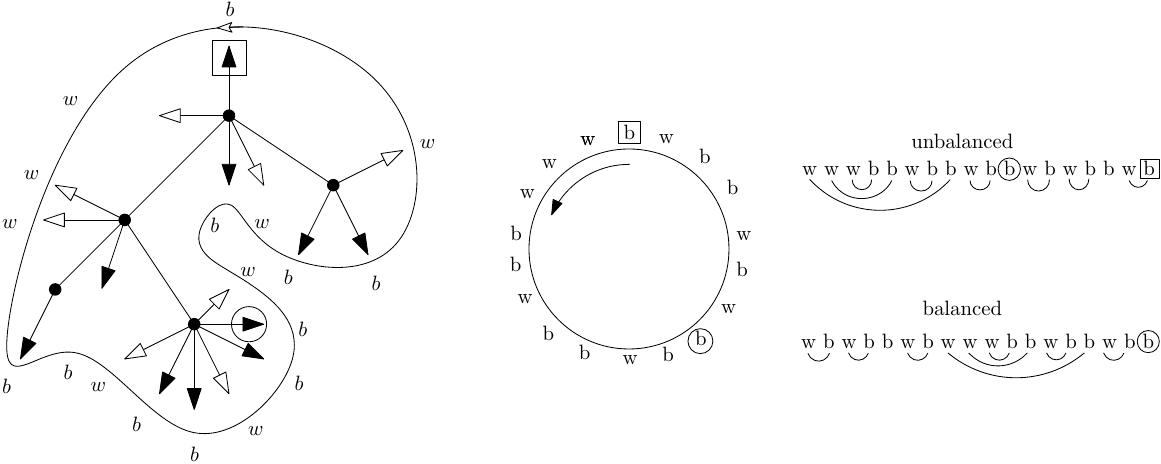}
    \caption{A Eulerian tree and the associated bracketing word with two different root choices, the square one being unbalanced and the circle one balanced. Leaves are represented by arrows and inner vertices by disks.}
    \label{fig:tree-word}
\end{figure}

\begin{thm}[\cite{Sch97}]
    For all $\tuple{d}$, there is a bijection $\close$ between balanced Eulerian trees of degree distribution $\tuple{d}$ and Eulerian maps of degree distribution $\tuple{d}$.
\end{thm}

Schaeffer's mappings are the following:

\noindent\textbf{Closure:}
Let $\gt \in \mathcal{T}(\tuple{d})$ be a balanced Eulerian tree.
\begin{enumerate}
    \item Going counterclockwise around the tree, merge each black leaf with the last unmerged white leaf seen. This is the same matching as the one done in the bracketing word.
    \item Merge the two free leaves and root the map on this edge oriented from the root leaf to the other.
\end{enumerate}
\begin{figure}[ht!]
    \centering
    \includegraphics[height=4cm]{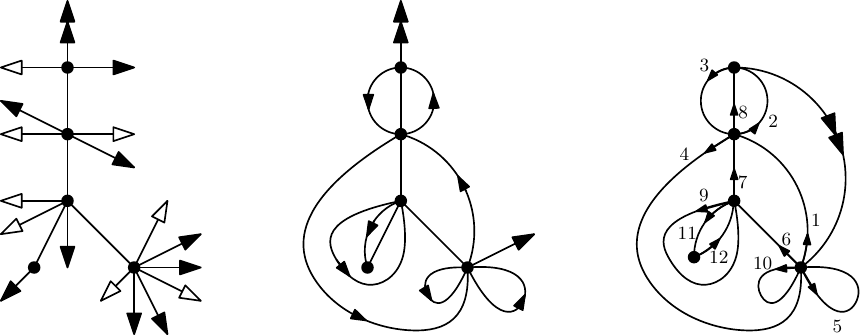}
    \caption{A Eulerian tree and its closure. Numbers and arrows on the map indicate the order and the direction in which the edges are explored during the opening.}
    \label{fig:tree closure}
\end{figure}
We denote by $\close(\gt)$ the result, called the \emph{closure} of $\gt$. See Figure~\ref{fig:tree closure} for an example. Observe that the closures of two conjugated balanced Eulerian trees are the same map up to the orientation of the root edge.

\begin{lemma}
    \label{lem:forget orient}
    Let $\gt$ be a balanced Eulerian tree and denote by $\gt'$ its balanced conjugate. Then $\sigma(\gt')$ is the map obtained by changing the orientation of the root of $\sigma(\gt)$.
\end{lemma}
\begin{proof}
    Recall that the words of $\gt$ and $\gt'$ can be written as $\omega(\gt)=\omega_1b\omega_2b$ and $\omega(\gt')=\omega_2b\omega_1b$. So a letter $w$ is matched to the same $b$ in both word, and so the closure merges a given white leaf with the same black leaf on both trees. The only difference occurs in the last step, in which the first free leaf seen changes and so opposite orientations are chosen for the root edge.
\end{proof}

\smallskip
We now describe the reverse mapping $\open$. A \emph{bridge} is an edge whose deletion disconnects the map. The \emph{successor} of a half-edge is defined as the half-edge that follows it when turning in counterclockwise direction around its end vertex, ignoring leaves. This definition is illustrated in Figure~\ref{fig:successor}.

\begin{figure}[ht!]
    \centering
    \includegraphics[height=2cm]{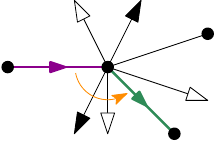}
    \caption{The successor of the purple half-edge is the green one.}
    \label{fig:successor}
\end{figure}

\noindent\textbf{Opening:}
Let $\gm \in \mathcal{E}(\tuple{d})$ be a Eulerian map. Cut the root edge in two, turn each half into a black leaf and mark its origin as the root of the future tree. Then set the current half-edge, $e$, to the successor of the root edge. While the map is not a tree, repeat the following steps:
\begin{enumerate}
    \item If $e$ is not a bridge then cut it, put a white leaf at its origin and a black leaf at its end.
    \item Set $e$ to its successor.
\end{enumerate}

The result, $\open(\gm)$, is called the \emph{opening} of $\gm$. Observe that the set of inner edges of $\open(\gm)$, i.e. the edges of $\gm$ that were not cut, form a spanning tree of $\gm$. We denote by $\optree(\gm)$ this spanning tree, which we call the \emph{opening tree} of $\gm$.

Finally, let us state properties of these bijections that will allow us to convert the \ref{eq:vertex} and the \ref{eq:face} to identities on balanced Eulerian trees, which are easier to manipulate.
\begin{lemma}[\cite{Sch97}]
\label{lem:corresp}
    Duality and Schaeffer's bijection $\close$ induce a correspondence between the following elements:
    \[\arraycolsep=0pt
    \begin{array}{ccccc}
        \textbf{Bipartite map } & \xleftrightarrow{duality} & \textbf{Eulerian map} & \xleftrightarrow{\ \close \ } & \textbf{Balanced Eulerian tree} \\
        \hline
        \text{Face of degree } 2i & \leftrightarrow & \text{Vertex of degree } 2i & \leftrightarrow &  \text{Inner vertex of degree } 2i \\
        \text{Vertex of degree } k & \leftrightarrow & \text{Face of degree } k & \leftrightarrow & \text{Black leaf } \\
        \text{Edge } uv & \leftrightarrow & \text{Edge between } u^* \text{and } v^* & \leftrightarrow & \text{Edge or pair of matched leaves} \\
    \end{array}\]
\end{lemma}

\begin{proof}
    Every correspondence is clear except for the black leaves of the balanced Eulerian tree and the faces of the Eulerian map. This correspondence was  already noticed by Schaeffer in \cite{Sch97}, we remind it here. 

    Consider the closure of a balanced Eulerian tree. When a non-free black leaf is merged with a white leaf, a face is closed. Associate this face to the black leaf. When the two free leaves are merged together, two faces are closed, associate the outer face to the root leaf and the other face to the other free leaf.
\end{proof}

\section{A bijection for the marked vertex identity}
\label{sec:bijection vertex}

In this section, we give a bijective proof of the \ref{eq:vertex} 
\[4(f(\tuple{d})-1)B(\tuple{d}) = \sum_{\tuple{s}+\tuple{t}=\tuple{d}} v(\tuple{s})v(\tuple{t})B(\tuple{s})B(\tuple{t}). \]

\subsection{Interpretation in terms of trees}

Using the correspondence of Lemma~\ref{lem:corresp}, the right-hand side of the \ref{eq:vertex} counts couples of balanced Eulerian trees $(\gt_1, \gt_2) \in \mathcal{T}(\tuple{s}) \times \mathcal{T}(\tuple{t})$ with $\tuple{s} + \tuple{t} = \tuple{d}$, each having a marked black leaf. For the left-hand side, observe that since $f(\tuple{d})$ is the number of inner vertices of a balanced Eulerian tree $\gt \in \mathcal{T}(\tuple{d})$, $f(\tuple{d}) -1$ is its number of inner edges. We then interpret the $4(f(\tuple{d})-1)$ factor as a marked oriented inner edge and a sign.

The idea will be to cut the big balanced Eulerian tree at the marked edge. The orientation of the edge will then be used to choose which of the resulting Eulerian trees is the first one, and the sign will be used to choose on which of its two free leaves the second resulting tree will be rooted. The following lemma provides a way to convert a sign into a choice of a free leaf.

\begin{lemma}[Rerooting]
\label{lem:reroot}
    There is a bijection between balanced Eulerian trees with a marked black leaf and general Eulerian trees with a sign.
\end{lemma}
\begin{proof}
    Let $\gt$ be a balanced Eulerian tree with root leaf $r$ and a marked leaf $\ell$. Reroot $\gt$ at $\ell$ and set the sign to $+$ if $r$ is the first free leaf encountered when following the border of $\gt$ clockwise starting from $\ell$, and to $-$ otherwise. 

    The reverse construction consists in rerooting the tree on the first free leaf encountered when following the boundary of the tree in the direction given by the sign, and marking the former root leaf.
\end{proof} 


We can now build the bijection between balanced Eulerian trees with a marked oriented inner edge and a sign, and pairs of balanced Eulerian trees with a marked black leaf each. An example is illustrated in Figure~\ref{fig:bij v trees}.

\textbf{The merging $\varphi$: }
Let $(\gt_1, \ell_1)$ and $(\gt_2, \ell_2)$ be two balanced Eulerian trees with a marked black leaf. Let $\varepsilon_i$ be the sign obtained when applying the Lemma~\ref{lem:reroot} to $(\gt_i, \ell_i)$. Let $\gt$ be the tree obtained by merging the two marked leaves into an edge, $e$. Mark the edge $e$ and orient it from $\gt_1$ to $\gt_2$. Notice that for each inner vertex, its degree and the number of white leaves it is adjacent to remain the same. Thus $\gt$ is an unrooted Eulerian tree and its degree distribution is the sum of those of $\gt_1$ and $\gt_2$. Root $\gt$ at the first free leaf seen when following the border of the tree starting at the right of $\vec{e}$, going counterclockwise if $\varepsilon_1 = +$ and clockwise otherwise. This yields a balanced Eulerian tree $\gt$. Return $\varphi((\gt_1,\ell_1),(\gt_2,\ell_2)) \coloneqq (\gt, \vec{e}, \varepsilon_2)$. 

\textbf{The separation $\psi$:}
For the reverse mapping, let $(\gt, \vec{e}, \varepsilon)$ be a balanced Eulerian tree with a marked oriented inner edge and a sign. Starting at the right of $\vec{e}$, follow the border of $\gt$ until a free leaf is encountered. If it is the root leaf then set $\varepsilon'=+$, otherwise set $\varepsilon'=-$. Cut the edge $e$, put a black leaf at each of its ends and mark these leaves. Denote by $\gt_1$ the tree that contains the origin of $e$ and by $\gt_2$ the other. Again, since for each inner vertex its degree and the number of white leaves it is adjacent to remain the same, both trees are unrooted Eulerian trees. Root $\gt_1$ using $\varepsilon'$ and $\gt_2$ using $\varepsilon$ as described in Lemma~\ref{lem:reroot}. Return $\psi(\gt, \vec{e}, \varepsilon) \coloneqq ((\gt_1, \ell_1), (\gt_2, \ell_2))$ where $\ell_i$ is the marked leaf of $\gt_i$.

\begin{figure}[ht]
    \centering
    \includegraphics[page=2, height=5cm]{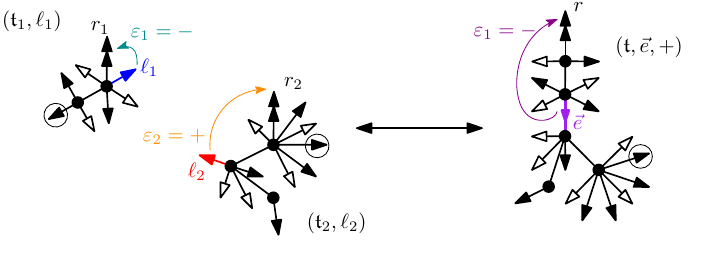}
    \caption{An example of the bijection for the marked vertex identity on balanced Eulerian trees. Double arrows are the root leaves and circled arrows are the other free leaves.}
    \label{fig:bij v trees}
\end{figure}

\begin{prop}
    The mappings $\varphi$ and $\psi$ are inverse of each other.
\end{prop}

We now have a bijective proof of the \ref{eq:vertex}, but it uses several intermediate step (through balanced Eulerian trees). In what follows, we investigate how to decompose bipartite maps directly, without transforming them into Eulerian trees.

\subsection{The bijection on bipartite maps}

Let us now go back to bipartite map and the \ref{eq:vertex}. First, we should understand to what kind of marked element the coefficient $4(f(\tuple{d})-1)$ corresponds to. On a Eulerian map $\gm$, it corresponded to a marked oriented inner edge of the opening $\open(\gm)$ of $\gm$, i.e. to an oriented edge in a special spanning tree $\optree(\gm)$ (the opening tree) of the Eulerian map. The following duality result allows us to transfer this to bipartite maps.

Let $\gm$ be a planar map and let $S \subset E(\gm)$ be a subset of its edges. The dual edge set of $S$ is defined as $S^* = \{e^* \mid e \in E(\gm)\setminus S\} \subset E(\gm^*)$. In the special case where $S$ is a spanning tree, the following result holds.

\begin{prop}[\cite{GraphTheory}]
\label{prop:tree dual}
    Let $\gm$ be a planar map and let $S$ be a spanning tree of $\gm$. Then $S^*$ is a spanning tree of the dual map $\gm^*$ of $\gm$.
\end{prop}

So if $\gm \in \cB(\tuple{d})$, the marked element of $\gm$ should be an oriented edge $e$ such that $e \notin \gt$ where $\gt = \optree(\gm^*)^*$ is a spanning tree of $\gm$. We claim that $\optree(\gm^*)^*$ is the breadth first search tree of $\gm$, which is built as follows. 

Let $\gm$ be a bipartite map. Initialize the current half-edge $\vec{e}$ to the root edge of $\gm$, the set of selected edges $T \coloneqq \varnothing$ and the set of cut edges $C \coloneqq \varnothing$. While $T \cup C \neq E(\gm)$, do as follows:

\begin{enumerate}
    \item[--] If $e \in T$, let $v$ be the end vertex of $\vec{e}$. Set $\vec{e}$ to the half-edge that follows $e$ when turning clockwise around $v$ and ignoring edges in $C$, oriented away from $v$ (this might be the reverse orientation of $\vec{e}$).
    \item[--] Else, if adding $e$ to $T$ does not create a cycle then set $T \coloneqq T \cup \{e\}$; else set $C \coloneqq C \cup \{e\}$. In both cases, set $\vec{e}$ to the next half-edge encountered when turning clockwise around the origin of $\vec{e}$ and ignoring edges in $C$.
\end{enumerate}
Denote by $\tree(\gm)$ the submap of $\gm$ with edge set $T$, called \emph{BFS tree} of $\gm$. By construction, $\tree(\gm)$ has no cycle. Furthermore, as $\gm$ is connected, when the procedure terminates, all edges have been seen. Thus, all vertices of $\gm$ have degree at least one in $\tree(\gm)$ since the first edge incident to a given vertex explored cannot create a cycle in $T$, so it is always added. Therefore, $\tree(\gm)$ is a spanning tree of $\gm$. We delay the proof that $\tree(\gm)$ is the dual of $\optree(\gm^*)$ to the next section (Proposition~\ref{prop:tree}).

\hfill

We now build a bijection between bipartite maps with degree distribution $\tuple{d}$, a marked oriented edge $\vec{e}$ that is not in their BFS tree and a sign; and couples of bipartite maps with a marked vertex each and respective degree distributions $\tuple{s}$ and $\tuple{t}$ such that $\tuple{s} + \tuple{t} = \tuple{d}$. We denote by $\cB^\bullet(\tuple{d})$ (resp. $\cB^\rightarrow(\tuple{d})$) the set of bipartite maps with degree distribution $\tuple{d}$ and a marked vertex (resp. a marked oriented edge that is not in their BFS tree).



\smallskip

\noindent\textbf{Procedure $\sew$ (Slit and Sew):}

Let $(\gm_1, v_1)\in \mathcal{B}^\bullet(\tuple{s})$ and $(\gm_2, v_2)\in \mathcal{B}^\bullet(\tuple{t})$ be two bipartite maps with a marked vertex each. Denote $\gt_i \coloneqq \tree(\gm_i)$. Their sewing $\sew((\gm_1, v_1), (\gm_2, v_2))$ is built as follows (see Figure~\ref{fig:v sew ex} for an example).
\begin{enumerate}
    \item \textbf{Slit.} For $i = 1, 2$, let $\pp_i$ be the unique path in $\gt_i$ between $v_i$ and the endpoint of the root of $\gm_i$ the farther from $v_i$. Let $k = \min(|\pp_1|, |\pp_2|)$. Slit the first $k$ edges of $\pp_i$ in $\gm_i$ starting from $v_i$, dividing each edge into two semi-edges. This creates in each map a ``hole" of length $2k$, bordered by semi-edges, which we will refer as the border of the map. 
    \item \textbf{Recolor.} If $v_1$ and $v_2$ have the same color, swap the colors in $\gm_2$ and set $\varepsilon = -$. Otherwise set $\varepsilon = +$.
    \item \textbf{Sew.} Sew together the borders of the two maps, positioning them such that $v_1$ is identified with the vertex that follows $v_2$ when following the border of ${\gm}_2$ with the ``hole'' to the right. Denote by $\gm$ the resulting map. Notice that, thanks to the recolor step, when two vertices are merged, they have the same color so the resulting coloring is well defined and proper.
    \item \textbf{Root.} If $|\pp_1| \neq |\pp_2|$, then exactly one of the root edges of the original maps $\gm_i$ was not slit, root the map $\gm$ on it. Otherwise, $|\pp_1| = |\pp_2|$, root the map $\gm$ on the edge opposite to $v_1v_2$ on the sewed cycle. In both cases, the orientation of the root edge is given by the coloring.  
    \item \textbf{Marking.} Mark the edge $\vec{e} \coloneqq v_1v_2$, oriented from $v_1$ to $v_2$.
    \item Return $\sew((\gm_1, v_1), (\gm_2, v_2)) \coloneqq (\gm, \vec{e}, \varepsilon)$.
\end{enumerate}


\begin{figure}[ht!]
    \centering
    \includegraphics[width=\textwidth]{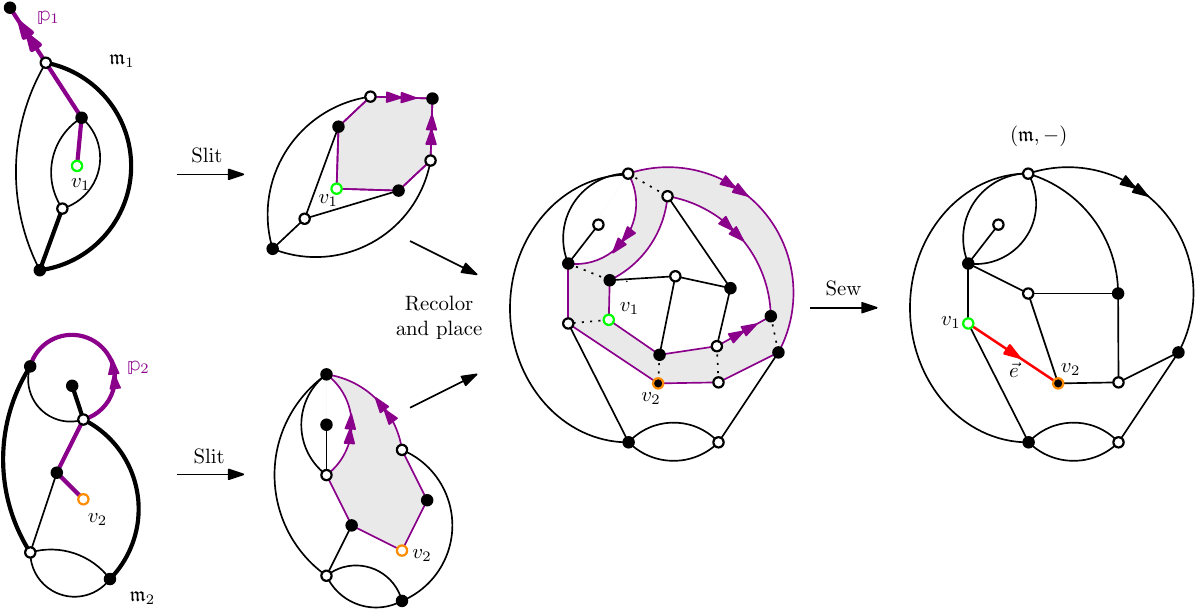}
    \caption{An example of the Slit and Sew procedure $\sew$. The bold edges form the spanning trees $\gt_i$ and the purple edges are the slit paths $\pp_i$.}
    \label{fig:v sew ex}
\end{figure}

The reverse mapping is the following.

\noindent\textbf{Procedure $\cut$ (Cut and Close):}

Let $(\gm, \vec{e}) \in \mathcal{B}^{\rightarrow}(\tuple{d})$ be a bipartite map with a marked oriented edge $\vec{e}$ that is not in its BFS spanning tree $\tree(\gm)$ and a sign $\varepsilon$. Its cut $\cut(\gm, \vec{e}, \varepsilon)$ is built as follows (see Figure~\ref{fig:v cut ex} for an example).
\begin{enumerate}
    \item Denote by $v_1$ the origin of $\vec{e}$ and $v_2$ its end. Let $\cc$ be the unique cycle formed by $e$ and edges of $\tree(\gm)$.
    \item \textbf{Cut.} Cut the map $\gm$ along $\cc$, and denote by $\gm_1$ the side to the left of $\vec{e}$ and by $\gm_2$ the other side. Transfer the marking $v_i$ to $\gm_i$. 
    \item \textbf{Close.} For $i = 1,2$, close the map $\gm_i$ by sewing together the two paths between $v_i$ and the vertex opposite to it in $\cc$.
    \item \textbf{Recolor.} If $\varepsilon = -$, swap the colors of $\gm_2$.
    \item \textbf{Root.} For $i = 1,2$, if $\gm_i$ inherited the root edge of $\gm$, root $\gm_i$ on it. Else root $\gm_i$ at the last edge of the path closed (opposite from $v_i$). 
    \item Return $\cut((\gm, \vec{e}, \varepsilon)) \coloneqq((\gm_1, v_1), (\gm_2, v_2))$.
\end{enumerate}

\begin{figure}[ht!]
    \centering
    \includegraphics[width=0.9\textwidth]{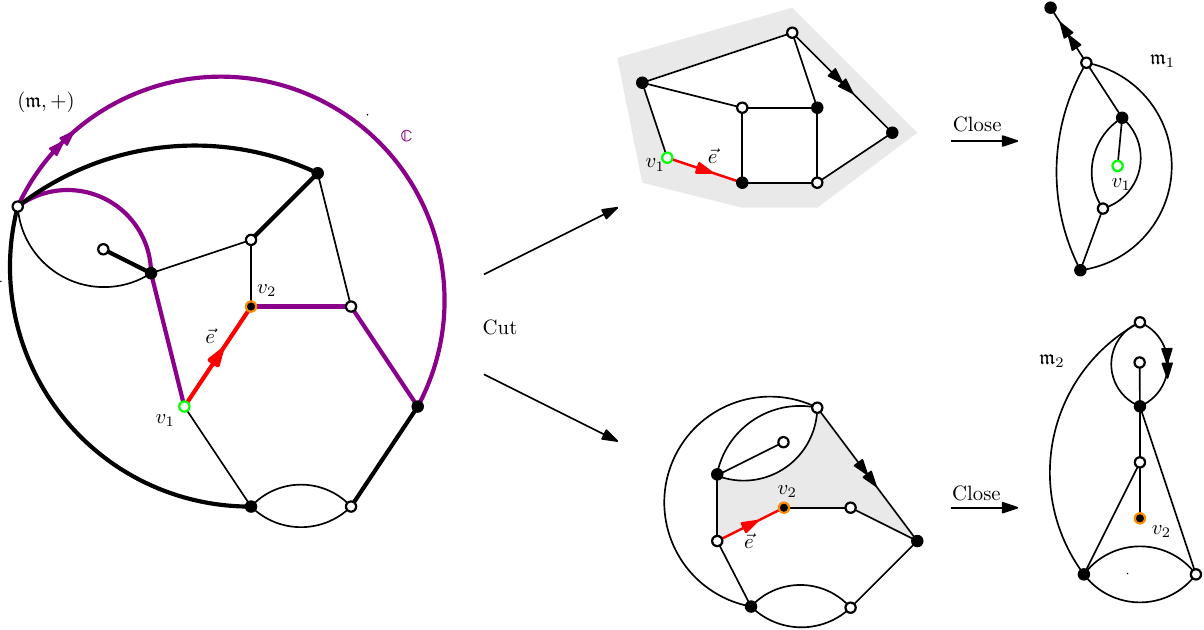}
    \caption{An example of the Cut and Close procedure $\cut$. The bold edges form the spanning tree $\gt$ and the purple edges are the cut cycle $\cc$.}
    \label{fig:v cut ex}
\end{figure}

\begin{thm}
    \label{thm:equiv}
    The mappings $\sew$ and $\cut$ are inverse of each other and induce a bijection between $\cB^\rightarrow(\tuple{d})\times\{\pm\}$ and $\bigcup_{\tuple{s}+\tuple{t}=\tuple{d}}\cB^\bullet(\tuple{s})\times\cB^\bullet(\tuple{t})$ for every degree distribution $\tuple{d}$. 
    
    Furthermore, they are equal to the mapping $\merge$ and $\sepa$ on the trees composed with duality and closure (denoted $\sigma$). In other words, the following diagram commutes:
    \[\arraycolsep=1pt\def\arraystretch{2}
    \begin{array}{ccccc}
        \bigcup_{\tuple{s}+\tuple{t}=\tuple{d}}\cB^\bullet(\tuple{s})\times\cB^\bullet(\tuple{t}) & \xleftrightarrow{duality} & \bigcup_{\tuple{s}+\tuple{t}=\tuple{d}}\cE^\bullet(\tuple{s})\times\cE^\bullet(\tuple{t}) & \xleftrightarrow{~\close~} & \bigcup_{\tuple{s}+\tuple{t}=\tuple{d}}\cT^\bullet(\tuple{s})\times\cT^\bullet(\tuple{t}) \\
        \sew \bigg\downarrow \bigg\uparrow \cut &&&& \merge \bigg\downarrow \bigg\uparrow \sepa \\
        \cB^\rightarrow(\tuple{d})\times\{\pm\} & \xleftrightarrow{duality} & \cE^\rightarrow(\tuple{d})\times\{\pm\} & \xleftrightarrow{~\close~} & \cT^\rightarrow(\tuple{d})\times\{\pm\} \\
    \end{array}
    \]
    where $\cE^\bullet$ is the set of Eulerian maps with a marked face, $\cE^\rightarrow$ is the set of Eulerian maps with a marked oriented edge in their opening tree, $\cT^\bullet$ is the set of balanced Eulerian trees with a marked black leaf, and $\cT^\rightarrow$ is the set of balanced Eulerian trees with a marked oriented inner edge.
    
\end{thm}

This will be a corollary of Propositions~\ref{prop:sew} and~\ref{prop:cut} which state that each of the procedures on the trees and on the bipartite maps have the same effect.


\subsection{The two pairs of mappings have the same result}
\label{sec:proof}

This section is dedicated to the proof of Theorem~\ref{thm:equiv}.

More precisely, let $(\gt_1^*, \ell_1)$ and $(\gt_2^*, \ell_2)$ be two balanced Eulerian trees with a marked black leaf. Let $(\gt^*, \vec{e^*}, \varepsilon)$ be the result of their merging. For $i=1, 2$, let $\gm_i^*$ be the closure $\close(\gt_i^*)$ of $\gt_i^*$ and let $f_i$ be the face the leaf $\ell_i$ is mapped to. Then let $\gm_i$ be the map which has $\gm_i^*$ as its dual (i.e. $\gm_i^{**}$ with its root edge orientation reversed) and $v_i \coloneqq f_i^*$. Let $\gm$ and $\gm^*$ be the maps obtained from $\gt$ in the same way. In the rest of this section, we show $\sew$ maps $((\gm_1, v_1), (\gm_2, v_2))$ to $(\gm, \vec{e}, \varepsilon)$ and $\cut$ maps $(\gm, \vec{e}, \varepsilon)$ to $((\gm_1, v_1), (\gm_2, v_2))$. In other words, we prove that the following diagram commutes: 
\begin{equation}
\label{eq:diagram}
    \def\arraystretch{2}
    \arraycolsep=3pt
\begin{array}{ccccc}
    ((\gm_1,v_1), (\gm_2, v_2)) & \xleftrightarrow{\text{ duality }} & ((\gm^*_1, v^*_1), (\gm^*_2, v^*_2)) & \xleftrightarrow{\text{closure }\close} & (\gt^*_1, \ell_1) + (\gt^*_2, \ell_2) \\
    \sew \bigg\downarrow \bigg\uparrow \cut &&&& \merge \bigg\downarrow \bigg\uparrow \sepa \\
    (M, \vec{e}, \varepsilon) & \xleftrightarrow{\text{ duality }} & (\gm^*, \vec{e^*}, \varepsilon) & \xleftrightarrow{\text{closure }\close} & (\gt^*, \vec{e^*}, \varepsilon)
\end{array}
\end{equation}

\subsubsection{The spanning trees are dual of each other.}

In this section, we prove that for any bipartite map $\gm$, the BFS tree $\tree(\gm)$ of $\gm$ is the dual of the opening tree $\optree(\gm^*)$ of $\gm^*$ (as set of edges of $\gm$ and $\gm^*$ respectively).

\begin{prop}
\label{prop:tree}
    For every bipartite map $\gm$, the BFS tree $\tau(\gm)$ of $\gm$ is the dual of the opening tree $\optree(\gm^*)$ of $\gm^*$.
\end{prop}

\begin{proof}
    Let $\gm$ be a bipartite map. We prove the result by running simultaneously the procedures $\tree$ and $\open$ on respectively $\gm$ and $\gm^*$, and showing by induction that the edges cut in $\gm^*$ by $\open$ are exactly the dual of those added to $T$ by $\tree$. 
    
    To make the proof easier, we consider for the construction of each tree $\tree(\gm)$ and $\optree(\gm^*)$ the slower procedure that does not ignore the cut edges (resp. leaves) but can take them as current edge. If a cut edge (resp. leaf) is the current edge, the slow procedures simply changes the current edge to the next edge turning around the vertex (clockwise for $\tree$ and counterclockwise for $\open$).

    Denote by $T_k$, $C_k$ and $e_k$ respectively the set $T$, the set $C$ and the current half-edge at the beginning of step $k$ in $\tree(M)$. Denote by $\gm^*_k$ the partially opened map of the beginning of step $k$ in $\open(\gm^*)$ and by $S_k$ the set of edges that have been cut in $\gm^*_k$. We show by induction on the step $k$ that, at every step, $T_k = \{e^* \mid e\in S_k\}$ and the current half-edge in $\open$ is the dual of $e_k$, with compatible orientation (see Figure~\ref{fig:next edge}).

    When the procedures start, $T_0 = S_0 = \varnothing$ and the current half-edges of the procedures $\tree$ and $\close$ are the root edges of respectively $\gm$ and $\gm^*$, which are dual to each other. 

    Assume that the statement holds at step $k$. 

    \begin{figure}[ht]
        \centering
        \includegraphics[page=4, height=2.7cm]{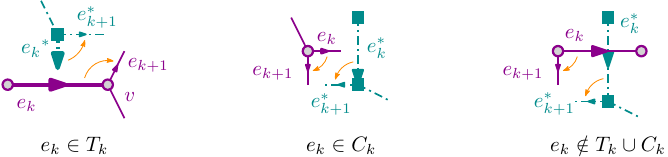}
        \caption{The next edge for the two procedure in the different cases. The bipartite map $\gm$ is in plain purple edge and round vertices, and the dual Eulerian map $\gm^*$ is in dashed teal edges and square vertices.}
        \label{fig:next edge}
    \end{figure}

    If $e_k \in T_k$, then $e_k^* \in S_k$ was cut. Thus the current edge of $\open$ is a leaf. The procedure $\open$ will only change the current edge to the next edge around the origin vertex of $e_{k+1}^*$ counterclockwise, while $\tree$ will change it to the next edge going clockwise around the face at the right of $e_k$. These two half-edges are dual to each other (see Figure~\ref{fig:next edge}, left). 

    Similarly, if $e_k \in C_k$, then ${{e}_k}^*$ is a bridge in $\gm^*_k$. Indeed, this means that $e_k$ forms a cycle $\cc$ with edges of $T_i$ for some $i<k$. Then as $T_i \subset T_k$, all the dual of the other edges in $\cc$ have been cut in $\gm^*_k$, so $e_k^*$ is a bridge. Thus $\tree$ takes for $\vec{e}_{k+1}$ the next edge around the origin of $\vec{e}_k$ clockwise, while $\open$ takes the next edge around the end of $e_k^*$, which is $e_{k+1}^*$ (see Figure~\ref{fig:next edge}, middle).

    Now assume $e_k \notin T_k\cup C_k$. If $e_k$ is not added to $T$, then it forms a cycle $\cc$ with edges in $T_k$. Then, by the same argument as above, $e_k^*$ is a bridge, and so it is not cut and not added to $S$. Therefore, $T_{k+1} =T_k$, $S_k = S_{k+1}$ and as above, the next current edges of the procedures are dual. Conversely, if $e_k^*$ is not added to $S$, it is a bridge in $\gm^*_k$ and so $e_k$ forms a cycle with edges of $T_k$ and is not added to $T$. Moreover, when this is not the case, $e_k$ is added to $T_{k+1}$ and $e_k^*$ is cut, so we have $T_{k+1} = T_k\cup\{e_k\} = \{e \mid e\in S_k\cup{e_k^*}=S_{k+1}\}$. Finally, in this case $e_{k+1}$ is the next edge when turning clockwise around the origin of $e_k$, and the new current edge of $\open$ is the edge following $e_k^*$ going counterclockwise around the face to its right, i.e. the dual of $e_{k+1}^*$ (see Figure~\ref{fig:next edge}, right).
\end{proof}

\subsubsection{The Merging and the Slit and Sew procedures have the same result.}

In this section, we prove the following statement:
\begin{prop}
    \label{prop:sew}
    For any bipartite maps $\gm_1$ and $\gm_2$ with a marked vertex each, the Slit and Sew procedure $\sew$ gives the same marked map as taking the dual of each map, opening them, merging the trees with $\merge$, closing the result and taking its dual.
\end{prop}

Let $\gm_1$ and $\gm_2$ be two maps with a marked vertex $v_i$. Recall the notations defined at the beginning of the section \eqref{eq:diagram}. 

The procedures $\sew$ and $\close\circ\merge\circ\open$ both cut some edges of the maps: $\close\circ\merge\circ\open$ cuts them perpendicularly into two half-edges while $\sew$ slits them into two semi-edges, and then rematch them differently. Observe that a semi-edge of a map corresponds to a half-edge of its dual map. We proceed by proving that both procedures return the same edge matchings.

First, let us show that when applying $\close\circ\merge\circ\open$ to $((\gm_1^*, v_1^*), (\gm_2^*, v_2^*))$, the set of half edges that are not matched with the same half edge in $\gm_i$ and in $\gm$ is contained in the sets $\{e^*\mid e\in \pp_i\}$ where $\pp_i$ is the path between $v_i$ and the root edge of $\gm_i$ in $\gt_i \coloneqq \tree(\gm_i)$. In other words, the only edges that may be changed by $\close\circ\merge\circ\open$ are the duals of the edges that are slit in the Slit step of $\sew$.

\begin{lemma}
\label{lem:match sew}
    Let $b$ and $w$ be two leaves of one of the $\gt_i^*$ that are matched together by $\close$ into an edge $e^*$ of $\gm^*_i$. If the dual edge $e$ of $e^*$ is not on the path $\pp_i$ between $v_i$ and the root edge in $\gt_i$ then $b$ and $w$ are still matched together in $\gt^*$.
\end{lemma}

\begin{proof}
    Without loss of generality, assume that $b$ and $w$ are leaves of $\gt_1$. 
    
    First notice that, since $\gt$ and $\gt^*$ are dual (as set of edges of respectively $\gm$ and $\gm^*$), $\pp_1$ is exactly the set of edges $e_j$ such that, denoting $b_j$ and $w_j$ respectively the black and white leaf that form $e_j^*$, the marked black leaf $\ell_1$ of $\gt_1$, is between $w_j$ and $b_j$ in the word $\omega(\gt_1)$ obtained by reading the boundary of $\gt_1$ (plus the root edge). Therefore, the border of $\gt_1$ reads $q_kw_kq_{k-1}\ldots w_1q_0w_0p_0\ell_1p_1b_1 \ldots b_{k+1}p_{k+1}b_{k+2}$ were $b_{k+1}$ and $b_{k+2}$ are the free leaves of $\gt_1$ and the $p_i$'s and $q_i$'s are correct bracketing words. When merging the two trees together, all those subwords are preserved so their closure remains the same.
\end{proof}

We can now prove the commutation stated in Proposition~\ref{prop:sew}. 

\begin{proof}[Proof of Proposition~\ref{prop:sew}]
    Lemma~\ref{lem:match sew} states that the edges modified by $\close\circ\merge\circ\open$ are contained in the duals of those slit in the Slit step of $\sew$. We now need to check that the semi-edge matching done by the Sew step of $\sew$ is the same as the half-edge matching from $\close\circ\merge\circ\open$, and that the choice of the root half-edge is the same. Figure~\ref{fig:bij v root} illustrates the parallel matchings in the two situations.

    Recall that closure and duality associate to each black leaf of $\gt_i^*$ a vertex of $\gm_i$. When drawing the tree $\gt_i^*$ on top of the bipartite map $\gm_i$ and considering leaves as half edges with their parents as their origin, the vertex one black leaf is mapped to is the vertex that lies to its right. Therefore, merging the two marked leaves of the trees corresponds to sewing together the two semi-edges following the vertices $v_i$ when following the border of the slit map $\gm_i$ with the cut to the right. 

    Writing the border of each trees $\gt_i^*$ as in the proof of Lemma~\ref{lem:match sew}, the border of the map obtained by closing the unchanged edges in $\gt^*$ reads  $w_{k_1} \ldots w_0 b_1' \ldots b_{{k_2}+1}'b_{{k_2}+2}'$ $w_{{k_2}}' \ldots w_0'b_1 \ldots b_{{k_1}+2}$, where the letters with a prime come from $\gt_2^*$ and the others from $\gt_1^*$. So the closure on the tree $\gt^*$ matches $w_j$ to $b_{j+1}'$ for $j \in \{0, \ldots \min(k_1, k_2+1)\}$ and $w_j'$ to $b_{j+1}$ for $j \in \{0, \ldots \min(k_2, k_1+1)\}$. This corresponds to sewing the borders of the two maps together starting from the vertices $v_i$, stopping when a root vertex is reached. If $k_1=k_2$, we are left with only two black leaves which become the root edge, at the position opposite to the edge $v_1v_2$ on the sewed cycle. Otherwise without loss of generality $k_1>k_2$ and the remaining leaves on the border read $w_{k_1} \ldots w_{{k_2}+2}b_{{k_2}+2} \ldots b_{{k_1}+2}$. The closure matches $w_j$ to $b_j$ for $j\in \{{k_2}+2\ldots k\}$ and $b_{{k_1}+1}$ to $b_{{k_1}+2}$ which forms the root. This corresponds to sewing back part of the path $\pp_1$ between $v_1$ and the root edge of $\gm_1$ so that the slit section has the same length as the other path. 

    Notice that the orientation of the root of $\gm$ is determined by the coloring of $\gm_1$, which is the same as rooting $\gt^*$ using $\varepsilon_1$ as root sign (recall from Lemma~\ref{lem:forget orient} that changing which free leaf is the root leaf only reverses the orientation of the root edge of the closure).
\end{proof}

\begin{figure}[ht]
    \centering    \includegraphics[width=0.95\textwidth]{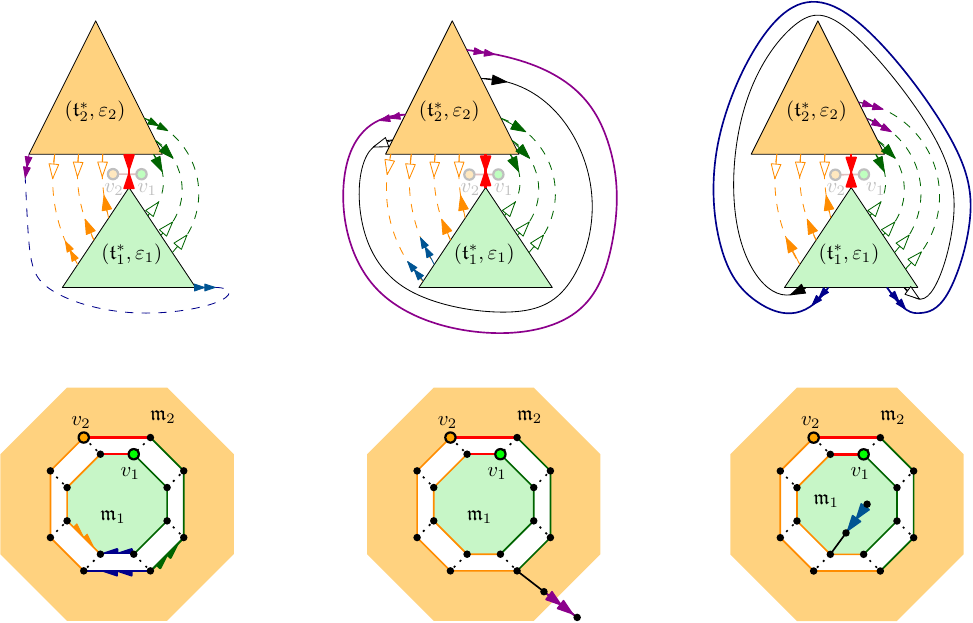}
    \caption{\textbf{Top:} The different possibilities for the map obtained by merging the trees $\gt_1^*$ and $\gt_2^*$ and performing the partial closure of the pairs of matched leaves which do not belong to a path $\pp_i$. The remaining leaves are represented by arrows and leaves that were free is either $\gt_1$ or $\gt_2$ by double arrows. The lines connect pairs matched leaves, and are dashed when the matching is different from the one in $\gt_i^*$. Left: $k_1=k_2$. Middle: $k_1<k_2$. Right: $k_1 > k_2$. \\
    \textbf{Bottom:} How the slit edges of $\gm_1$ and $\gm_2$ are sewed by $\sew$, depending on the relative length of the slit paths $\pp_i$.\\ Left: $|\pp_1|=|\pp_2|$. Middle: $|\pp_1|<|\pp_2|$. Right: $|\pp_1|>|\pp_2|$. \\
    A given bipartite map situation corresponds to the tree situation on top of it.}
    \label{fig:bij v root}
\end{figure}

\subsubsection{The Separation and the Cut and Close procedures have the same result.}

In this section, we prove the same result for the Cut and Close and the Separation proccesses.

\begin{prop}
    \label{prop:cut}
    For any bipartite map $\gm$ with a marked oriented edge, the Cut and Close procedure $\cut$ gives the same marked maps as taking the dual map, opening it with Schaeffer's bijection, separating the tree with $\sepa$, closing the two resulting trees and taking their dual.
\end{prop}

First let us define one more notion. Let $e \in \gt$ be an edge in a spanning tree $\gt$ of a map $\gm$. Its \emph{cocycle}, denoted $C(e)$, is the set of the edges of $\gm$ whose endpoints are in different connected components in $\gt-e$. Notice that the duals of the edges of a cocycle form a cycle in the dual map $\gm^*$. This implies that in a Eulerian map, all cocycles have even cardinality (this can also be obtained by double counting the edges in each connected component).

Let $\gm$ be a bipartite map with a marked oriented edge $\vec{e}$. As before, we use the notations of diagram~\eqref{eq:diagram}. As in the previous section, we will investigate which edges are modified by each procedure, and then prove that both procedures do the same half-edge matching.

\begin{lemma}
\label{lem:match cut}
    Let $\ell$ be a leaf of $\gt^*$. If $\ell$ is not in the cocycle of $e^*$ nor on an edge whose dual is in one of the two paths in $\gt$ between an extremity of $e$ and the root of $\gm$, then $\ell$ is matched to the same leaf in $\gt^*$ and in $\gt_i^*$.
\end{lemma}

\begin{proof}
    The set of leaves that form the edges of $C(e^*)$ and the two paths in $\gt$ between the extremities of $e$ and the root of $\gm$ is exactly the set of matched pairs of leaves $(b,w)$ such that a side of $e^*$ is seen between $w$ and $b$ when following the border of the tree $\gt^*$ counterclockwise. Indeed, the edges of $C(e^*)$ have their extremities in different connected components of $\gt^*-e^*$ by definition. Moreover, since $\gt$ and $\gt^*$ are dual (as sets of edges), the dual edge of an edge on a path from an extremity of $e$ to the root of $\gm$ in $\gt$ is an edge of $E(\gm^*)\setminus \gt^*$ that ``separates $e^*$ from the outer face of $\gm^*$''. That is, such an edge either is in $C(e^*)$ or is an edge of $\gm^*$ that surrounds the connected component of $\gt^*-e^*$ that does not contain the root edge of $\gm^*$. 
    
    Denote $U$ this set of leaves. Let $\ell$ be a leaf that is not in $U$ and let $p$ be the longest subword of $\omega(\gt)$ containing $\ell$ but no leaf in $U$ and not crossing $e$. Then $p$ is a correct bracketing word and since $e^*$ is not in $p$, all its leaves end up in the same tree $\gt_i$. Thus $p$ is closed the same way in $\gt$ and $\gt_i$. So $\ell$ is matched to the same leaf.
\end{proof}

\begin{lemma}
\label{lem:partial close}
    Consider the maps $\widetilde{\gm}^*_1$ and $\widetilde{\gm}^*_2$ obtained after merging all pairs of leaves in respectively $\gt^*_1$ and $\gt^*_2$ except those that were in $C(e^*)$. Then if $|C(e^*)| = 2k$, the number of remaining black leaves in each map is the following:
    \begin{itemize}
        \item $k+1$ in each map if the root of $\gm^*$ is in $C(e^*)$.
        \item $k$ in $\widetilde{\gm}^*_1$ and $k+1$ in $\widetilde{\gm}^*_2$ if the root of $\gm^*$ has both its extremities in $\widetilde{\gm}^*_1$
        \item $k$ in $\widetilde{\gm}^*_2$ and $k+1$ in $\widetilde{\gm}^*_1$ if the root of $\gm^*$ has both its extremities in $\widetilde{\gm}^*_2$
    \end{itemize}
    Furthermore, when following the border of one of the $\widetilde{\gm}_i^*$ counterclockwise one sees all the white leaves then all the black leaves.
\end{lemma}
\begin{proof}
    Through the procedure, an edge of $C(e^*)$ become a white leaf in one map $\widetilde{\gm}_i$ and a black one in the other, except if it is $e$ or the root edge of $\gm^*$ in which case it gives a black leaf to each map. The numbers given then follow from the fact that the trees $\gt_i^*$ are Eulerian trees, so each must have 2 more black leaves than white leaves.

    For the last statement, for a pattern $wb$ to appear, the black leaf should come from the edge $e^*$, as otherwise the two leaves would not have been in $C(e^*)$. Since there is only one extremity of $e^*$ in each $\widetilde{\gm}_i^*$, there can be only one occurrence of $wb$.
\end{proof}

We now have all the ingredients to prove Proposition~\ref{prop:cut}.

\begin{proof}[Proof of Proposition~\ref{prop:cut}]
    Lemma~\ref{lem:match cut} states that the edges modified by $\close\circ\sepa\circ\open$ are contained in the duals of those cut in Cut step of $\cut$. All that remains is to check that the semi-edge matching done by the Close step of $\cut$ is the same as in $\close\circ\sepa\circ\open$, and that the root choice is also the same. Without loss of generality, lets consider $\gt_1^*$. 
    
    First assume that $\gt_1^*$ has at most one free leaf of $\gt^*$ and consider its border after closing the edges that are not in $C(e^*)$ (note that the path to the root is contained in $C(e^*)$ here). According to the Lemma~\ref{lem:partial close}, the border reads as $w_{k-1} \ldots w_1b_1 \ldots b_{k+1}$, with $b_1$ denoting the black leaf from $e$. So $w_j$ is matched with $b_j$ for $j<k$ and $b_k$ is matched with $b_{k+1}$, forming the root edge. This corresponds to closing the cycle starting at $v_1$ and rooting the map on the last edge of the path. 

    Now assume both free leaves of $\gt^*$ are in $\gt_1^*$. The border $\gt_1^*$ after partial closure then reads $w_{k+r} \ldots w_{k+1}$ $w_k \ldots w_1$ $b_1 \ldots b_kb_{k+1} \ldots b_{k+r+2}$ where the letters with indices lesser or equal to $k$ are from $C(e^*)$ and the other are from the path to the root. The later are matched back as in $\gm$ and the former are matched with $w_j$ on $b_j$, which correspond to closing the cycle starting from $v_1$ and keeping the root edge of $\gm$. 
\end{proof}

\section{A bijection for the marked face identity}
\label{sec:bijection face}

In this section, we give a bijective proof of the \ref{eq:face}:
\[(f(\tuple{d})-1)(f(\tuple{d})-2)B(\tuple{d}) = \sum_{\tuple{s}+\tuple{t}=\tuple{d}} (f(\tuple{s})-1)v(\tuple{t})(v(\tuple{t})-1)B(\tuple{s})B(\tuple{t}).\]

\subsection{Interpretation in trees}

Using the correspondence from Lemma~\ref{lem:corresp}, the \ref{eq:face} implies a bijection between balanced Eulerian trees $\gt \in \mathcal{T}(\tuple{d})$ with two marked edges (left-hand side) and couples of trees $(\gt_1, \gt_2) \in \mathcal{T}(\tuple{s}) \times \mathcal{T}(\tuple{t})$ with $\tuple{s} + \tuple{t} = \tuple{d}$, the first one having a marked edge and the second one two marked black leaves (right-hand side). This is very similar to the interpretation of the \ref{eq:vertex}, except that the signs and orientations are missing.

We start with a trick that allows us to turn the marked edge of $\gt_1$ into a marked oriented edge. Observe that Lemma~\ref{lem:forget orient} allows us to work on unrooted Eulerian trees without compromising the interpretation in terms of bipartite maps. We therefore forget the roots of $\gt$ and $\gt_2$, which divides both sides of the equation by $2$. Then for $\gt_1$, we orient the marked edge by following the border of the tree clockwise and orienting the edge in the direction it is first seen, unless the second free leaf was seen before in which case we use the reverse orientation. We can then forget this root too as it is recoverable from the oriented edge. 

The bijection on tree is then the following:

\textbf{Merging:}
Let $\gt_1$ be an unrooted Eulerian tree with a marked oriented edge $\vec{e}$ and $\gt_2$ be an unrooted Eulerian tree with two marked black leaves $\ell_1$ and $\ell_2$. The Eulerian tree $\gt$ is built as follows. 
\begin{enumerate}
    \item Cut edge $e$ and denote by $h_1$ and $h_2$ the black leaves corresponding to its origin and to its end respectively.
    \item For $i = 1,2$, merge $h_i$ and $\ell_i$ together and mark the edge obtained as $e_i$.
\end{enumerate}

\textbf{Separation:}
Let $\gt$ be an unrooted Eulerian tree with two marked edges $e_1$ and $e_2$. We separate it into two trees $\gt_1$ and $\gt_2$ as follows.
\begin{enumerate}
    \item Cut both edges $e_i$, putting a black leaf at each of their ends. This cuts $\gt$ into three parts, one having two leaves from the edges $e_i$, which we will denote $\gt_2$, and the other two having one leaf for an edge $e_i$ each. Denote by $h_i$ the end of $e_i$ on the side with one marked leaf and by $\ell_i$ the end on the side with two marked leaves (so in $\gt_2)$.
    \item Merge $h_1$ and $h_2$ together, mark the edge and orient it from $h_1$ to $h_2$. Let $\gt_1$ be the result.
\end{enumerate}

\begin{figure}[ht]
    \centering
    \includegraphics[height=3.5cm]{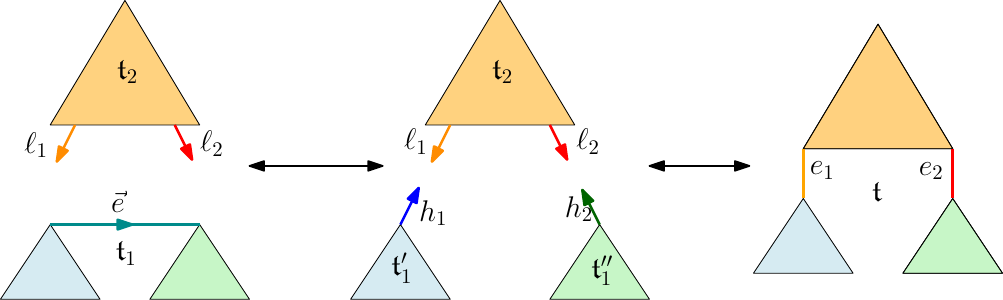}
    \caption{A schematic representation of the marked face bijection on trees.}
    \label{fig:bij f trees}
\end{figure}

\subsection{Translation to bipartite maps}

Notice that the bijection on trees can be seen as applying thrice the bijection of the previous section. More precisely, to merge the trees one applies $\sepa$ on $\gt_1$ to cut it in two, then applies twice $\merge$ to glue each of the two trees obtained to $\gt_2$.

\[
\def\arraystretch{1.7}
\begin{array}{rl}
(\gt_1, \vec{e}) + (\gt_2, \ell_1, \ell_2) &\xrightarrow{\sepa\text{ to }(\gt_1, \vec{e})} (\gt_1', h_1) + (\gt_1'', h_2) + (\gt_2, \ell_1, l_2)\\
&\xrightarrow{\merge\text{ to } (\gt_1', h_1) \text{ and } (\gt_2, l_1)} (\gt_1'', h_2) + (\gt_2', h_2, \vec{e_1})\\
&\xrightarrow{\merge\text{ to } (\gt_1'', h_2) \text{ and } (\gt_2', l_2)} (T, \vec{e_1}, \vec{e_2}) 
\end{array}
\]

Therefore, this second bijection on bipartite maps can be described as applying thrice the marked vertex bijection in the same way as on trees. 

The marked face bijection can also be described directly on bipartite map with slit and sew operations, but is it quite tedious so we choose to omit it. Figures~\ref{fig:face sew ex} and \ref{fig:face cut ex} provide examples of these two mappings.

\begin{figure}[ht!]
    \centering
    \includegraphics[width=0.88\textwidth]{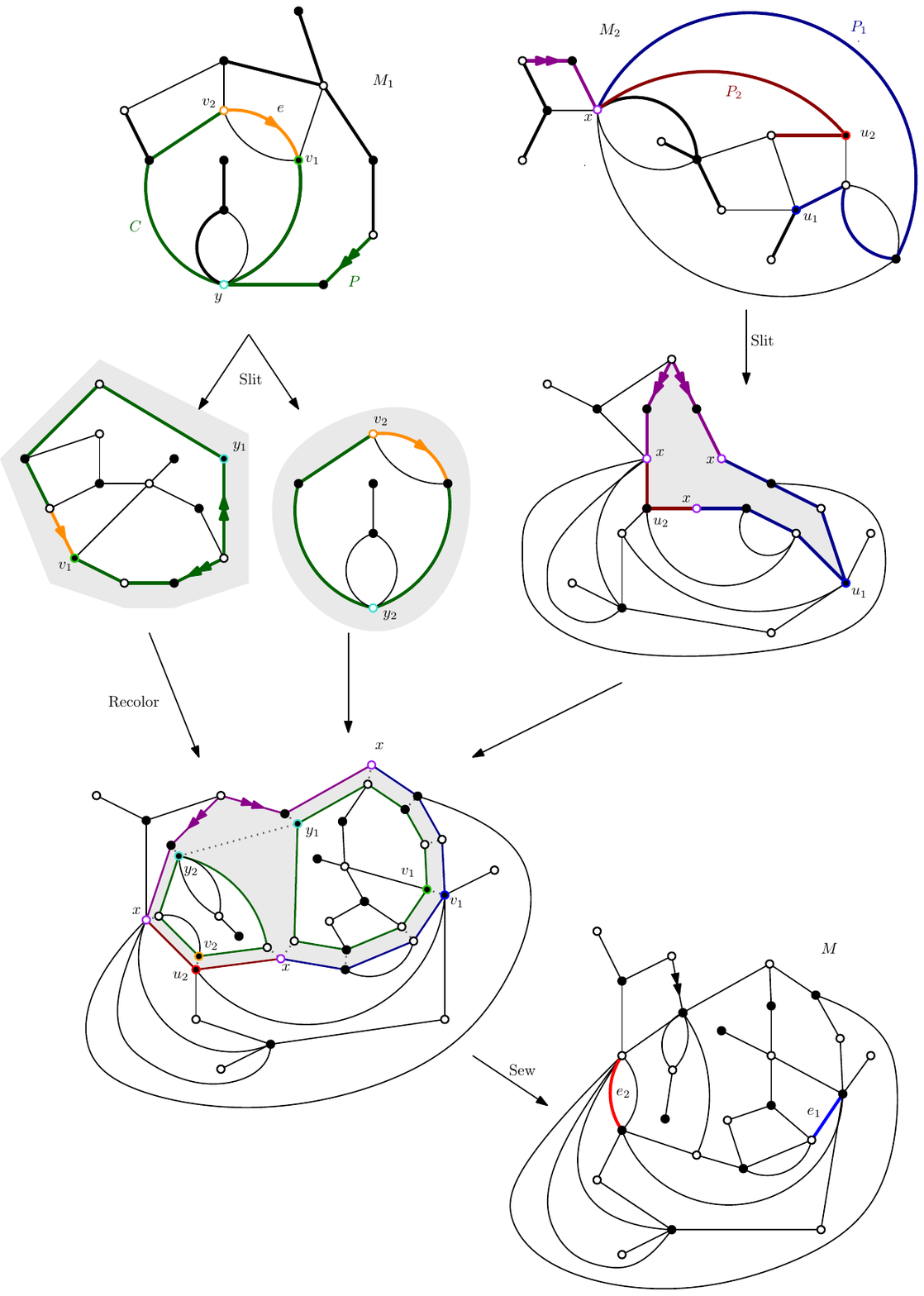}
    \caption{An example of the second Slit and Sew mapping. The bold edges are the spanning trees. The colored edges are the slit paths.}
    \label{fig:face sew ex}
\end{figure}

\begin{figure}[ht!]
    \centering
    \includegraphics[width=\textwidth]{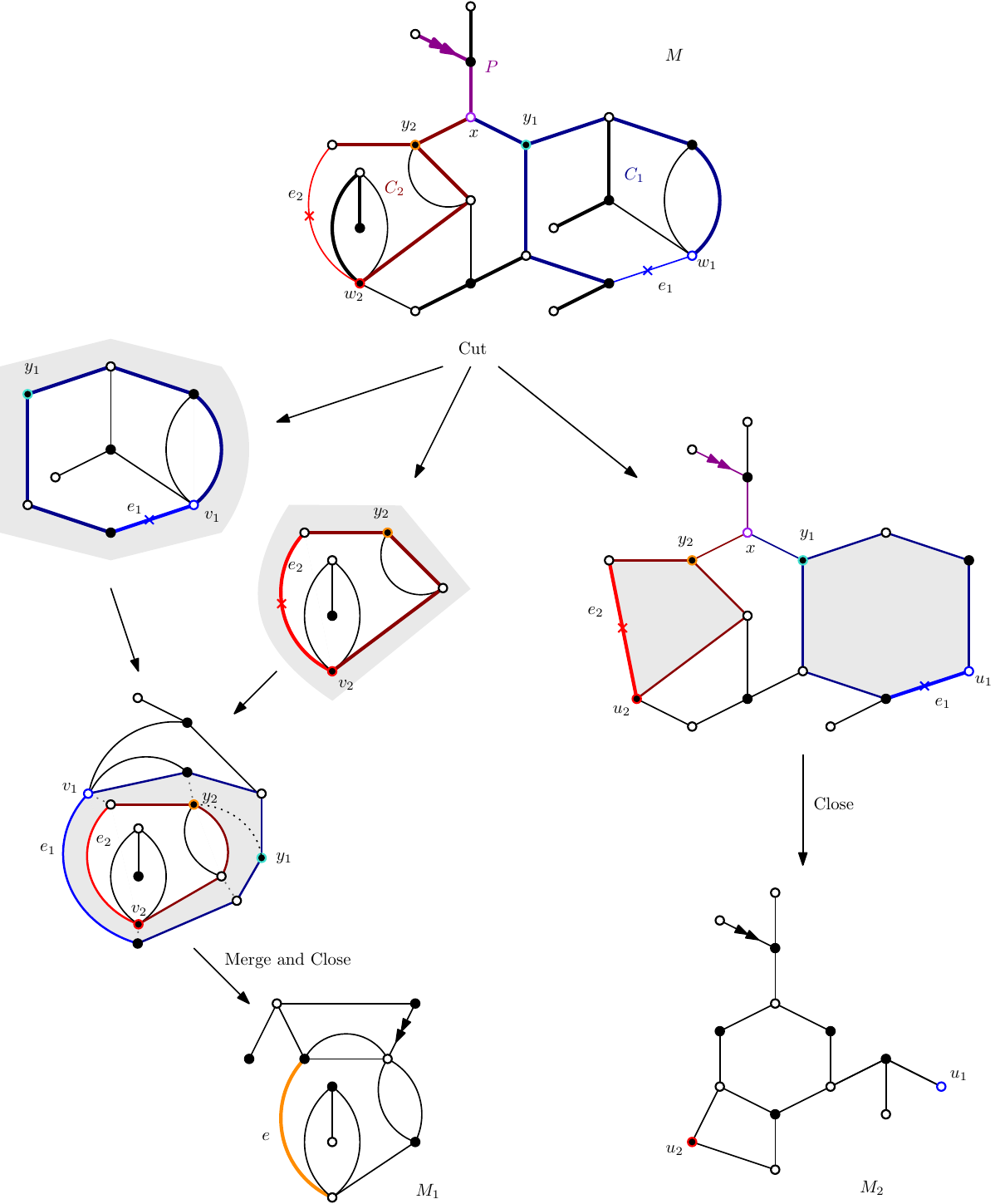}
    \caption{An example of the second Cut and Close mapping. The bold edges are the spanning trees. The blue and red edges form the two cut cycles and the purple one the paths to them.}
    \label{fig:face cut ex}
\end{figure}


\section*{Acknowledgements}

The author wishes to thank Baptiste Louf for suggesting the problem and for useful discussions, as well as Mireille Bousquet-Mélou for her insightful comments.

\newpage

\bibliographystyle{alpha}
\bibliography{biblio}

\end{document}